\documentclass[review]{elsarticle}
\usepackage{lineno,hyperref}
\usepackage{mathrsfs}
\usepackage{amsfonts}
\usepackage{amssymb}
\usepackage{amsmath}
\usepackage{hyperref}% webpage
\usepackage{amsthm}
\usepackage{float}
\usepackage{cases}
\usepackage{arydshln}
\usepackage[usenames,dvipsnames]{color}
\usepackage{graphicx}%
\newtheorem{thm}{Theorem}[section]
\newtheorem{cor}[thm]{Corollary}
\newtheorem{lem}[thm]{Lemma}
\newtheorem{prob}[thm]{Problem}

\newtheorem{defin}[thm]{Definition}

\newtheorem{eg}[thm]{Example}

\newcommand{\tr}{\text{Tr}}
\newcommand{\diag}{\text{diag }}
\modulolinenumbers[5]
\allowdisplaybreaks[4]
\journal{J. Math. Anal. Appl.}

\begin{document}
	
	\begin{frontmatter}
		
		\title{When $D$-companion matrix meets incomplete polynomials}
		
		\author{Teng Zhang\fnref{1footnote}}
		\address{School of Mathematics and Statistics, Xi'an Jiaotong University, Xi'an 710049, China}
		\fntext[1footnote]{Email: teng.zhang@stu.xjtu.edu.cn}

		\begin{abstract}
			 In this paper,  we provide a simple proof of a generalization of the Gauss-Lucas theorem. By using methods of $D$-companion matrix, we get the majorization relationship between the zeros of  convex combinations of incomplete polynomials and an origin polynomial. Moreover, we prove that  the set of all zeros of all convex combinations of incomplete polynomials  coincides with the closed convex hull of zeros of the original polynomial. The location of zeros of convex combinations of incomplete polynomials is determined.
		\end{abstract}
		
		\begin{keyword}
			Gauss-Lucas theorem, $D$-companion matrix, incomplete polynomials, majorization relationship of zeros
		\end{keyword}
		
	\end{frontmatter}
	
	%\linenumbers
	
	\section{Introduction} 
	The most basic and important theorem linking the zeros of the derivative of a polynomial to the zeros of the polynomial are attributed to Gauss and Lucas.
	\begin{thm}(Gauss-Lucas' Theorem).\cite[p. 18]{BE95}\label{GL}
		Let $A_n(z)$ be a monic and complex polynomial of degree $n$. Then all the zeros of the derivative $A_n'(z)$ of $A_n(z)$ are contained in the closed convex hull of the set of zeros of $A_n(z)$.
	\end{thm}
	In \cite{DB08}, D\'iaz-Barrero and Egozcue introduced the concept of incomplete polynomial. A definition follows.
	\begin{defin}
		(Incomplete polynomials). Let $z_1,z_2,\ldots,z_n$ be $n$, not necessarily distinct, complex numbers. The incomplete polynomials of degree $n-1$, associated with $z_1,z_2,\ldots,z_n$, are the polynomials $g_k(z)$, $1\le k\le n$, given by
		\begin{eqnarray*}
			g_k(z)=\prod_{j=1,j\neq k}^n (z-z_j).
		\end{eqnarray*}
	\end{defin}
Let $A_n(z)$ be a monic polynomial of degree $n$  whose zeros are $z_1,z_2,\ldots,z_n$ (i.e., the original polynomial). That is,
\[
A_n(z)=\prod_{j=1}^{n}(z-z_{j}).
\] 
Let $\gamma=(\gamma_1,\gamma_2,\ldots,\gamma_n)$ be nonnegative real numbers such that $\sum\limits_{k=1}^n\gamma_k=1$. The corresponding linear convex combination of incomplete polynomials is denoted $A_n^\gamma(z)=\sum\limits_{k=1}^n \gamma_kg_k(z)$. Note that $A_n^\gamma(z)$ is a monic polynomial of degree $n-1$. Clearly, the derivative of $A_n(z)$, normalized to be monic is then one of such convex linear combinations.

 D\'iaz-Barrero and Egozcue \cite{DB08} generalized the Gauss-Lucas theorem \ref{GL} to $A_n^\gamma(z)$.
 \begin{thm}\cite[Theorem 2]{DB08}\label{thm1}
 	Let $z_1,z_2,\ldots,z_n$ be $n$, not necessarily distinct, complex numbers.  Then, the polynomial $A_n^\gamma(z)=\sum\limits_{k=1}^n\gamma_kg_k(z)$ has all its zeros in or on the convex hull $H(z_1,z_2,\ldots,z_n)$ of the zeros of $A_n(z)=\prod\limits_{j=1}^{n}(z-z_j)$.
 \end{thm}
In this paper, we provide a simple proof of the above result. 

In \cite{Sch03}, Schmeisser first get the majorization relationship between the critical points and zeros of a polynomial.
\begin{thm}\cite[Corollary 4]{Sch03}\label{thm2}
	 Let $z_1,z_2,\ldots,z_{n}$ denote the zeros of $A_n(z)$ listed in descending order of modulus and $w_1,w_2,\ldots,w_{n-1}$ denote the critical points of $A_n(z)$ listed in descending order  of modulus, $w_n:=0$. Then
	 \[
	 	\prod_{j=1}^k\left|w_j\right|\le\prod_{j=1}^k\left|z_j\right|, \text{ for } 1\le k\le n. 
	 \]
\end{thm}
We will see Theorem \ref{thm2} can be generalized the following result.
\begin{thm}\label{thm3}
		 Let $z_1,z_2,\ldots,z_{n}$ denote the zeros of $A_n(z)$ listed in descending order of modulus and $w_1,w_2,\ldots,w_{n-1}$ denote the zeros of $A^\gamma_n(z)$ listed in descending order  of modulus, $w_n:=0$. Then
	\[
	\prod_{j=1}^k\left|w_j\right|\le\prod_{j=1}^k\left|z_j\right|, \text{ for } 1\le k\le n. 
	\]
\end{thm}
In \cite{P07}, Pereira totally verified the  result of Cheung and Ng \cite{CN06}.
\begin{thm}\cite[Theorem 2.1]{P07}
	Let $n\ge 2$ and $z_1,z_2,\ldots,z_n\in \mathbb{C}$. Let $p(z)=\prod\limits_{j=1}^n(z-z_j)$ and $p_a(z)=\prod\limits_{j=1}^{n}(z-\left|z_j\right|)$. Let $w_1,w_2,\ldots,w_{n-1}$ denote the critical points of $p(z)$ listed in descending order of modulus and $v_1,v_2,\ldots,v_{n-1}$ denote the critical points of $p_a(z)$ listed in descending order. Then 
	\[
	\prod_{j=1}^k\left|w_j\right|\le\prod_{j=1}^k\left|v_j\right|, \text{ for } 1\le k\le n-1. 
	\]
\end{thm}
We also have the similar result on  linear convex combinations of incomplete polynomials.
\begin{thm}\label{thm4}
		Let $n\ge 2$ and $z_1,z_2,\ldots,z_n\in \mathbb{C}$. Let $A_n(z)=\prod\limits_{j=1}^n(z-z_j)$ and $B_n(z)=\prod\limits_{j=1}^{n}(z-\left|z_j\right|)$. Let $w_1,w_2,\ldots,w_{n-1}$ denote the zeros of $A_n^\gamma(z)$ listed in descending order of modulus and $v_1,v_2,\ldots,v_{n-1}$ denote the zeros of $B_n^\gamma(z)$ listed in descending order. Then 
	\[
	\prod_{j=1}^k\left|w_j\right|\le\prod_{j=1}^k\left|v_j\right|, \text{ for } 1\le k\le n-1.
	\]
\end{thm}
Before proceeding, let us fix some notation. Let $A$ be an $n\times n$ complex matrix. The spectral radius of $A$ is denoted by $\rho(A):=\max\limits_{1\le j\le n}\{\left|\lambda_{j}\right|\}$, where $\lambda_{1},\ldots,\lambda_n$ are all eigenvalues of $A$. We use $s_j(A)$ to denote the $j$th largest singular values of $A$. The transpose of $A$ is denoted by $A^T$. The conjugate transpose of $A$ is denoted by $A^*$ and $\left|A\right|:=(A^*A)^{\frac{1}{2}}$. If all the eigenvalues of $A$ are nonnegative, we use $\lambda_{j}(A)$ to denote the $j$th largest eigenvalue of $A$. We use diag $(d_1,d_2,\ldots,d_n)$ to denote a diagonal matrix with main diagonal entries $d_1,d_2,\ldots,d_n$. Let $A,B$ be two Hermitian matrices. Then $A\ge B$ means $A-B$ is positive semidefinite.
\section{proof of Theorem \ref{thm1} }
\emph{Proof of Theorem \ref{thm1}: } Clearly,
\begin{equation*}
	\dfrac{A_n^\gamma(z)}{A_n(z)}=\sum\limits_{j=1}^n\dfrac{\gamma_j}{z-z_j}=\sum\limits_{j=1}^n\dfrac{\gamma_j(\bar{z}-\bar{z_j})}{\left|z-z_j\right|^2}.
\end{equation*}
Let $w$ be any zero of $A_n^{\gamma}(z)$ but not be a zero of $A_n(z)$. Then
\begin{equation}\label{e4}
	\sum\limits_{j=1}^n\dfrac{\gamma_j(\bar{w}-\bar{z_j})}{\left|w-z_j\right|^2}=0
\end{equation}
Rearranging (\ref{e4}) gives
\begin{equation}\label{e5}
	\sum\limits_{j=1}^n\dfrac{\gamma_j\bar{w}}{\left|w-z_j\right|^2}=\sum\limits_{j=1}^n\dfrac{\gamma_j\bar{z_j}}{\left|w-z_j\right|^2}.
\end{equation}
Taking conjugates into (\ref{e5}), we have
\[
w\left\lbrace\sum\limits_{j=1}^n\dfrac{\gamma_j}{\left|w-z_j\right|} \right\rbrace =\sum\limits_{j=1}^n\dfrac{\gamma_jz_j}{\left|w-z_j\right|^2}.
\]
Let $\beta_j=\dfrac{\tfrac{\gamma_j}{\left|w-z_j\right|^2}}{\sum\limits_{j=1}^n\tfrac{\gamma_j}{\left|w-z_j\right|}}$. Then $\sum\limits_{j=1}^n\beta_j=1$ and $w=\sum\limits_{j=1}^n\beta_jz_j.$ If $w=z_k$ for some $k$, the result is obvious. \hfill $\square$
\section{Proof of Theorem \ref{thm3} and Theorem \ref{thm4}}
We need the following lemmas.
\begin{lem}\cite[Theorem 1.2]{CN10}\label{le1}
	Let $A$ be an $n\times n$ matrix with characteristic polynomial $p(z)=\prod\limits_{j=1}^n(z-z_j)$ and $q(z)$ be a monic polynomial of degree $n-1$ given by
	\[
	\dfrac{q(z)}{p(z)}=\sum\limits_{j=1}^n\dfrac{\lambda_j}{z-z_j}.
	\]
	Then  the characteristic polynomial of the matrix $D(I-\Lambda J)$ is $zq(z)$, where $D=\diag (z_1,z_2,\ldots,z_n)$, $\Lambda=\diag(\lambda_1,\lambda_2,\cdots,\lambda_n)$ and $J$ is the $n\times n$ all one matrix.
\end{lem}
\begin{defin}
	The matrix $D(I-\Lambda J)$ in Lemma \ref{le1} is called a $D-$companion matrix of $zq(z)$.
\end{defin}
\begin{lem}\cite[p. 43]{Bha97}\label{le2}
	Let $A$ be an $n\times n$ matrix with singular values $s_1\ge s_2\ge \ldots\ge s_n$ and eigenvalues $\lambda_1,\lambda_2,\ldots,\lambda_n$ arranged in such a way that $\left|\lambda_1\right|\ge \left|\lambda_2\right|\ge \ldots\ge \left|\lambda_n\right|$. Then
	\[
	\prod_{j=1}^{k}\left|\lambda_j\right|\le \prod_{j=1}^ks_j, \text{ for }1\le k\le n.
	\]
\end{lem}
\emph{Proof of Theorem \ref{thm3}: }
Let  $w_1,w_2,\ldots,w_{n-1}$ be the zeros of $A_n^\gamma(z)$.  By using Lemma \ref{le1}, we know the characteristic polynomial of the matrix $D(I-\Lambda J)$ is $zA_n^\gamma(z)$, where $D=\diag (z_1,z_2,\ldots,z_n),\Lambda=\diag(\gamma_1,\gamma_2,\cdots,\gamma_n)$ and $J$ is the $n\times n$ all one matrix. 

It's easy to see that the eigenvalues of
$D(I-\Lambda J)$ concide with the eigenvalues of $D(I-\Lambda^{\frac{1}{2}}J\Lambda^{\frac{1}{2}})$.

Since $I-\Lambda^{\frac{1}{2}}J\Lambda^{\frac{1}{2}}$ is a contraction, $\left|(I-\Lambda^{\frac{1}{2}}J\Lambda^{\frac{1}{2}})D^*\right|^2=D(I-\Lambda^{\frac{1}{2}}J\Lambda^{\frac{1}{2}})^2D^*\le DD^*=\left|D^*\right|^2$. The L\"owner-Heinz inequality (see \cite[p. 115]{Bha97}) yields $\left|(I-\Lambda^{\frac{1}{2}}J\Lambda^{\frac{1}{2}})D^*
\right|$ $\le \left|D^*\right|$. By using Weyl's inequalities (see \cite[p. 62]{Bha97}), we get
\begin{equation}\label{e6}
	\lambda_j(\left|(I-\Lambda^{\frac{1}{2}}J\Lambda^{\frac{1}{2}})D^*
	\right|)\le \lambda_j(\left|D\right|).
\end{equation}
By using Lemma \ref{le2} and (\ref{e6}), we have
\[
\begin{aligned}
	\prod_{j=1}^{k}\left|\lambda_j(D(I-\Lambda J))\right|&= \prod_{j=1}^k\left|\lambda_j(D(I-\Lambda^{\frac{1}{2}}J\Lambda^{\frac{1}{2}}))\right|\\
	&\le \prod_{j=1}^k\lambda_j(\left| D(I-\Lambda^{\frac{1}{2}}J\Lambda^{\frac{1}{2}})\right|) \\
	&=\prod_{j=1}^k\lambda_j(\left|(I-\Lambda^{\frac{1}{2}}J\Lambda^{\frac{1}{2}})D^*\right|)\\
	&\le \prod_{j=1}^k\lambda_j(\left|D\right|), \text{ for }1\le k\le n.
\end{aligned}
\]
That is, $\prod\limits_{j=1}^k\left|w_j\right|\le\prod\limits_{j=1}^k\left|z_j\right|, \text{ for } 1\le k\le n. $ \hfill$\square$

\emph{Proof of Theorem \ref{thm4}: } Let $D^{\frac{1}{2}}=\diag(\sqrt{z_1},\sqrt{z_2},\ldots,\sqrt{z_n})$. (We choose the square root which is either a non-negative real number or a complex number with positive imaginary part.) Let $A=D^{\frac{1}{2}}(I-\Lambda J)D^{\frac{1}{2}}, A$ is similar to $B=D(I-\Lambda J)$ so its eigenvalues are $\{w_1,w_2,\ldots,w_{n-1},0\}$. Similarly, $\left|A\right|=(D^{\frac{1}{2}})^*(I-\Lambda J)D^{\frac{1}{2}}$ is similar to $\left|D\right|(I-\Lambda J)$ which has eigenvalues $\{v_1,v_2,\ldots,v_{n-1},0\}$. Thus, the singular values of $A$ are $ \{v_1,v_2,\ldots,v_{n-1},0 \}$. Theorem \ref{thm4} now follows from Lemma \ref{le2}.    \hfill$\square$
\section{Some majorization results}
\begin{lem}\cite[p. 42]{Bha97}
	Let $\phi$ be a function such that $\phi(e^t)$ is convex and monotone increasing in $t$. Suppose that $x=(x_1,\ldots,x_n),y=(y_1,\ldots,y_n)\in \mathbb{R}^{n}_+$ and 
	\[
	\prod_{j=1}^{k}x_j\le\prod_{j=1}^{k}y_j,  \text{ for }1\le k\le n.
	\]
	Then
	\[
	\sum_{j=1}^{k}\phi(x_j)\le \sum_{j=1}^{k}\phi(y_j),  \text{ for }1\le k\le n.
	\]
\end{lem}
So, it immediately follows from \ref{thm3} and \ref{thm4} that
\begin{thm}
	Let $\phi$ be a function such that $\phi(e^t)$ is convex and monotone increasing in $t$. Let $z_1,z_2,\ldots,z_{n}$ denote the zeros of $A_n(z)$ listed in descending order of modulus and $w_1,w_2,\ldots,w_{n-1}$ denote the zeros of $A^\gamma_n(z)$ listed in descending order  of modulus, $w_n:=0$. Then
	\[
	\sum_{j=1}^{k}\phi(\left|w_j\right|)\le \sum_{j=1}^{k}\phi(\left|z_j\right|),  \text{ for }1\le k\le n.
	\]
Specially, for $\phi(t)=t^p, p\ge 1$, we have
	\[
\sum_{j=1}^{k}\left|w_j\right|^p\le \sum_{j=1}^{k}\left|z_j\right|^p , \text{ for }1\le k\le n.
\]
\end{thm}
\begin{thm}
	Let $\phi$ be a function such that $\phi(e^t)$ is convex and monotone increasing in $t$. Let $n\ge 2$ and $z_1,z_2,\ldots,z_n\in \mathbb{C}$. Let $A_n(z)=\prod\limits_{j=1}^n(z-z_j)$ and $B_n(z)=\prod\limits_{j=1}^{n}(z-\left|z_j\right|)$. Let $w_1,w_2,\ldots,w_{n-1}$ denote the zeros of $A_n^\gamma(z)$ listed in descending order of modulus and $v_1,v_2,\ldots,v_{n-1}$ denote the zeros of $B_n^\gamma(z)$ listed in descending order. Then
	\[
	\sum_{j=1}^{k}\phi(\left|w_j\right|)\le \sum_{j=1}^{k}\phi(\left|v_j\right|),  \text{ for }1\le k\le n-1.
	\]
	Specially, for $\phi(t)=t^p, p\ge 1$, we have
	\[
	\sum_{j=1}^{k}\left|w_j\right|^p\le \sum_{j=1}^{k}\left|v_j\right|^p , \text{ for }1\le k\le n-1.
	\]
\end{thm}
\section{The zeros of convex combinations of incomplete polynomials}
 Notice that the derivative of $A_n(z)$:
$A_n'(z)=\sum\limits_{k=1}^ng_k(z)$ is an elementary symmetric polynomial of degree $n-1$, where $g_k(z)=\prod\limits_{j\neq k}^n(z-z_j)$. Gauss-Lucas Theorem \ref{GL} tell us that all the zeros of $A_n'(z)$ are contained in the closed convex hull $H(z_1,\ldots,z_n)$.
The $2$th derivative of $A_n(z)$: $A_n^{(2)}(z)=\sum\limits_{1\le i< j\le n}g_{ij}(z)$ is an elementary symmetric polynomial of degree $n-2$, where $g_{ij}(z)=\tfrac{\prod\limits_{k=1}^n(z-z_k)}{(z-z_i)(z-z_j)}$. Applying the Gauss-Lucas Theorem \ref{GL} to $A_n'(z)$, we know  all the zeros of $A_n^{(2)}(z)$ are also contained in the closed convex hull $H(z_1,\ldots,z_n)$. In general, for $1\le k\le n$, the $k$th derivative of $A_n(z)$: $A_n^{(k)}(z)=\sum\limits_{1\le i_1<i_2<\ldots<i_k\le n}g_{i_1i_2\ldots i_k}(z)$ is an elementary symmetric polynomial of degree $n-k$, where $g_{i_1i_2\ldots i_k}(z)=\tfrac{\prod\limits_{k=1}^n(z-z_k)}{(z-z_{i_1})(z-z_{i_2})\cdots(z-z_{i_k})}$. It's easy to get the following result.
\begin{thm}
	All the zeros of $A_n^{(k)}(z)$ are contained in the closed convex hull $H(z_1,\ldots,z_n)$.
\end{thm}

Theorem \ref{thm1} shows that all the zeros of any convex combination $\sum\limits_{k=1}^n\gamma_k g_k(z)$ of $\{g_k\}_{1\le k\le n}$  are contained in the closed convex hull $H(z_1,z_2,\ldots,z_n)$ . This movitates us to ask the following problem.
\begin{prob}\label{p1}
Are	all the zeros of any convex combination $\sum\limits_{k=1}^nr_{ij}g_{ij}(z)$ of $\{g_{ij}(z)\}_{1\le i<j\le n}$  contained in the closed convex hull $H(z_1,z_2,\ldots,z_n)$?
\end{prob}
However, the answer to Problem \ref{p1} is false.
\begin{eg}
	Let $p(z)=z^2(z-i)^2,g(z)=\dfrac{1}{3}\left(z^2+z(z-i)+(z-i)^2 \right)=z^2-iz-\tfrac{1}{3}=(z-\tfrac{1}{2}i)^{2}-\tfrac{1}{12} $. Clearly, the zeros of $g(z)$ are $\tfrac{1}{2}i\pm\tfrac{1}{2\sqrt{3}}\notin H(0,i)$.
\end{eg}
In fact, the following result is true.
\begin{thm}
All the zeros of $\sum\limits_{k=1}^n\gamma_k\sum\limits_{i,j\neq k}^{n}g_{ij}$ are contained in the closed convex hull $H(z_1,\ldots,z_n)$, where $\sum\limits_{k=1}^n\gamma_k=1,\gamma_k\ge 0, 1\le k\le n$.
\end{thm}
\begin{proof}
	 \[\begin{aligned}
	 		A_n^\gamma(z)&=\sum\limits_{k=1}^n\gamma_kg_k(z)=\sum_{i=1}^n\gamma_k\prod\limits_{i=1,i\neq k}^n(x-x_i),\\
	 		A_n^\gamma(z)'&=\sum_{k=1}^n\gamma_k\sum_{j=1,j\neq k}^n\dfrac{\prod\limits_{i=1,i\neq k}^n(x-x_i)}{x-x_j}\\
	 		&=\sum_{k=1}^n\gamma_k\sum_{i,j\neq k}^{n}g_{ij}(z)\\
	 		&=\sum_{k=1}^n\sum_{i,j\neq k}^{n}\gamma_kg_{ij}(z),
	 \end{aligned}
	 \]
	 where $\sum\limits_{i,j\neq k}^{n}g_{ij}(z)=(z-z_2)\cdots(z-z_{k-1})(z-z_{k+1})\cdots(z-z_n)+(z-z_1)(z-z_3)\cdots(z-z_{k-1})(z-z_{k+1})\cdots(z-z_n)+\cdots+(z-z_1)\cdots(z-z_{k-1})(z-z_{k+1})\cdots(z-z_{n-1}) $.
	  
	  By using Gauss-Lucas Theorem, the convex hull of the zeros of $A_n^\gamma(z)$ contains all the zeros of $\sum\limits_{k=1}^n\sum\limits_{i,j\neq k}^{n}\gamma_kg_{ij}$. By using \ref{thm1}, we know the the convex hull of the zeros of $A_n(z)$ contains all the zeros of $A_n^\gamma(z)$. So the theorem is proved.
\end{proof}

Next, we ask that
\begin{prob}\label{p2}
	Let $g(z)$ be a monic polynomial of degree $n-1$ and the zeros of $g(z)$ be contained in  the convex hull $H(z_1,z_2,\ldots,z_n)$ of the set of zeros of $A_n(z)$. Whether $ g(z) $ can be expressed as some convex combination of incomplete polynomials?
\end{prob}
The following example shows that the answer to Problem \ref{p2} is false.
\begin{eg}
	Let $f(z)=z(z-1)(z-i)$ and $g(z)=z(z-1)$. Clearly, $\{0,\tfrac{1}{2}\}\subset H(0,1,i)$. We asumme that
	\[
	g(z)=\gamma_1(z-1)(z-i)+\gamma_2z(z-i)+\gamma_3(z-1)(z-i),
	\]
	where $\gamma_k\ge 0,1\le k \le 3$.
	By using method of equal coefficient, we get
	\[
	\left\lbrace \begin{aligned}
		\gamma_1+\gamma_2+\gamma_3&=1\\
		(1+i)\gamma_1+i\gamma_2+\gamma_3&=\dfrac{1}{2}\\
		i\gamma_1&=0
	\end{aligned}\right. .
	\]
	It implies that 
		\[
	\left\lbrace \begin{aligned}
		\gamma_1&=0\\
		\gamma_2&=\dfrac{1}{4}+\dfrac{1}{4}i\\
		\gamma_3&=\dfrac{3}{4}-\dfrac{1}{4}i\\
	\end{aligned}\right. .
	\]
	This contradicts our hypothesis $\gamma_2,\gamma_3\ge 0$.
\end{eg}
\begin{thm}
	Let $S$ be the set of all the zeros of all convex combinations of incomplete polynomials, that is, $S=\{z : A_n^\gamma (z)=0, \text{ for some }\gamma \}$. Then $S=H(z_1,z_2,\ldots,z_n)$.
\end{thm}
\begin{proof}
	Clearly, by using Theorem \ref{thm1}, we know $S\subset H(z_1,z_2,\ldots,z_n)$.
	
	Suppose $a\in H(z_1,\ldots,z_n)$, then \begin{eqnarray}\label{sh}
		\sum\limits_{i=1}^n t_i(a-z_i)=0,
	\end{eqnarray}
	 where $t_i\ge 0,i=1,\ldots,n$ and $\sum\limits_{i=1}^nt_i=1$.  
	
	Now take the complex conjugate of both sides and multiply each term by $(a-z_i)/(a-z_i)$ into (\ref{sh}) to get
	\[
	\sum\limits_{i=1}^n t_i \dfrac{|a-z_i|^2}{a-z_i}=0.
\]
	
	Now pick $\gamma_i= \dfrac{t_i|a-z_i|^2}{\sum\limits_{j=1}^n t_j|a-z_j|^2}$ and then $a$ is a zero of $A_n^\gamma$. That is, $H(z_1,z_2,\ldots,z_n)\subset S $.
\end{proof}
\section{The location of zeros of convex linear combinations of incomplete polynomial}
\begin{thm}\label{thm5}
	Let  $p(z)=\prod\limits_{j=1}^{n}(z-z_{j})$ be a  polynomial of degree $n$ and $q(z)$ be a monic polynomial of degree $n-1$ given by
	\begin{eqnarray*}
		\dfrac{q(z)}{p(z)}=\sum_{j=1}^{n}\dfrac{\lambda_{j}}{z-z_{j}}.
	\end{eqnarray*}
	Then the characteristic polynomial of the matrix $M=D(I-\Lambda J)+z_{n}\Lambda J$ is $q(z)$, where $D,\Lambda$ are diagonal matrices formed $z_{1},\ldots,z_{n-1}$ and $\lambda_{1},\ldots,\lambda_{n-1}$ respectively, and $J$ is the $(n-1)\times (n-1)$ all one matrix.
\end{thm}
\begin{proof} Let $e=(1,1,\ldots,1)^T\in\mathbb{R}^n$. We calculate that 
	\begin{align*}
		\det(zI-M)&= \det\left(zI-D\left(I-\Lambda J\right)-z_{n}\Lambda J\right) 
		\\&= \det\left((zI-D)+\Lambda\left(D-z_{n}I\right)J\right) 
		\\&= \det\left((zI-D)+\Lambda\left(D-z_{n}I\right)ee^T\right)\\&= \Big(\det(zI-D)\Big)\cdot \left(1+e^T(zI-D)^{-1} \Lambda\left(D-z_{n}I\right)e\right) 
		\\&= \Big(\det(zI-D)\Big)\cdot \left(1+e^T(zI-D)^{-1} \Lambda\left(D-z_{n}I\right)e\right)\\&=\left(\prod_{j=1}^{n-1}(z-z_j)\right)\cdot \left(1+\sum_{k=1}^{n-1}\frac{\lambda_{k}(z_k-z_n)}{z-z_k}\right)\\&= \left(\prod_{j=1}^{n-1}(z-z_j)\right)\cdot \left(\lambda_{n}+\sum_{k=1}^{n-1}\frac{\lambda_{k}(z-z_n)}{z-z_k}\right)\qquad (\text{Since $q(z)$ is monic, $\sum_{j=1}^n\lambda_{j}=1$.})\\&=\left(\prod_{j=1}^{n}(z-z_j)\right)\cdot\left(\sum_{k=1}^{n}\dfrac{\lambda_{k}}{z-z_{k}} \right) \\
		&=q(z).
	\end{align*}
\end{proof}
\begin{lem}\label{le3}\cite[p. 146]{HJ13}
	Let the eigenvalues of a given $A\in M_n$ be $\lambda_{1},\ldots,\lambda_{n}$. Then
	\[
	\max\limits_{i=1,\ldots,n}\left|\lambda_i-\dfrac{\tr\  A}{n}\right|\le \sqrt{\dfrac{n-1}{n}}\left(\tr\  A^*A-\dfrac{\left|\tr\  A\right|^2}{n}\right)^{\frac{1}{2}}. 
	\] 
\end{lem}
\begin{thm}(Ger\v{s}gorin).\label{thm6}\cite[p.388]{HJ13} Let $A=[a_{ij}]\in M_n$. let
	\[
	R_i'(A)=\sum_{j\neq i}\left|a_{ij}\right|,\ i=1,2,\ldots,n
	\]
	denote the deleted absolute row sums of $A$, and consider the $n$ Ger\v{s}gorin discs
	\[
	\{z\in\mathbb{C}:\left|z-a_{ii}\right|\le R_i'(A)\},\ i=1,\ldots,n.
	\]
	The eigenvalues of $A$ are in the union of Ger\v{s}gorin discs
	\[
	G(A)=\bigcup\limits_{i=1}^n\{z\in \mathbb{C}:\left|z-a_{ii}\right|\le R_i'(A)\}.
	\]
\end{thm}
\begin{thm}\label{thm7}
	Let $z_1,z_2,\ldots,z_n$ be $n$, not necessarily distinct, complex numbers.  Then all the zeros of the polynomial $A_n^\gamma(z)=\sum\limits_{k=1}^n\gamma_kg_k(z)$ lie in the disk
	\[\left|z-\frac{\sum\limits_{j=1}^{n}(1-\gamma_j)z_j}{n-1}\right|\le \sqrt{\frac{n-2}{n-1}}\left(\sum_{j=1}^{n-1}\left|(1-\gamma_j)z_j
	+\gamma_{j}z_n\right|^2+(n-2)\sum\limits_{j=1}^{n-1}\gamma_j^2\left|z_n-z_j\right|^2-\frac{\left|\sum\limits_{j=1}^{n}(1-\gamma_j)z_j\right|^2}{n-1} \right)^{\frac{1}{2}} .
	\]
\end{thm}
\begin{proof}
	By using Theorem \ref{thm5}, we know the characteristic polynomial of the matrix $M=D(I-\Lambda J)+z_{n}\Lambda J=(m_{ij})$ is $A_n^\gamma(z)$, where $D,\Lambda$ are diagonal matrices formed $z_{1},\ldots,z_{n-1}$ and $\gamma_{1},\ldots,\gamma_{n-1}$ respectively, and $J$ is the $(n-1)\times (n-1)$ all one matrix. That is,
	\begin{eqnarray}\label{e9}
		M=\left(\begin{matrix}
			(1-\gamma_1)z_1+\gamma_{1}z_n&\gamma_{1}(z_n-z_1)&\ldots&\gamma_{1}(z_n-z_1)\\
			\gamma_{2}(z_n-z_2)&(1-\gamma_2)z_2+\gamma_{2}z_n&\ldots&\gamma_{2}(z_n-z_2)\\
			\vdots&\vdots&\ddots&\vdots\\
			\gamma_{n-1}(z_n-z_{n-1})&\gamma_{n-1}(z_n-z_{n-1})&\ldots&(1-\gamma_{n-1})z_{n-1}+\gamma_{n-1}z_n
		\end{matrix} \right). 
	\end{eqnarray}
	It's easy to calculate that
	\[\begin{aligned}
		\tr M^*M&=\sum\limits_{i,j=1}^n\left|m_{ij}\right|^2=\sum_{j=1}^{n-1}\left|(1-\gamma_j)z_j
		+\gamma_{j}z_n\right|^2+(n-2)\sum\limits_{j=1}^{n-1}\gamma_j^2\left|z_n-z_j\right|^2,\\
		\tr M&=\sum\limits_{j=1}^{n-1}[(1-\gamma_j)z_j+\gamma_{j}z_n]=\sum\limits_{j=1}^{n} (1-\gamma_j)z_j .
	\end{aligned}
	\]
	Then applying Lemma \ref{le3} to $M$ in (\ref{e9}), we complete the proof of theorem .
\end{proof}
\begin{cor}
	Let  $p(z)=\prod\limits_{j=1}^{n}(z-z_{j})$ be a  polynomial of degree $n$ and $p'(z)$ be the derivative of $p(z)$. Then all the zeros of $p'(z)$ lie in the disk
	\[
	\left|z-\dfrac{\sum\limits_{j=1}^nz_j}{n}\right|\le \sqrt{\frac{n-2}{n-1}}\left(\sum_{j=1}^{n-1}\left|\dfrac{n-1}{n}z_j
	+\dfrac{1}{n}z_n\right|^2+\dfrac{n-2}{n^2}\sum\limits_{j=1}^{n-1}\left|z_n-z_j\right|^2-\dfrac{n-1}{n^2}\left|\sum\limits_{j=1}^{n}z_j\right|^2 \right)^{\frac{1}{2}}.
	\]
\end{cor}
\begin{proof}
	Take $\gamma_{j}=\frac{1}{n},1\le j\le n$ in Theorem \ref{thm3}.
\end{proof}
By using (\ref{e9}) and  Theorem \ref{thm6}, we have
\begin{thm}
	Let $z_1,z_2,\ldots,z_n$ be $n$, not necessarily distinct, complex numbers.  Then, all the zeros of the polynomial $A_n^\gamma(z)=\sum\limits_{k=1}^n\gamma_kg_k(z)$ lie in union of the discs
	\[
	\bigcup\limits_{j=1}^n\{z\in \mathbb{C}:\left|z-[(1-\gamma_j)z_j+\gamma_{j}z_n]\right|\le (n-2)\gamma_{j}\left|z_n-z_j\right|\}.
	\]
\end{thm}

	\section*{Acknowledgments} I heartily thank my advisor Minghua Lin for encouraging me to work on  relationship between the zeros of incomplete polynomials and the original polynomial and discussing with me frequently. I  also would like to thank the referee for several valuable suggestions and
	comments.

\end{document}